\documentclass[reqno,11pt]{amsart}
\makeatletter
\let\origsection=\section 
\def\section{\@ifstar{\origsection*}{\mysection}}
\def\mysection{\@startsection{section}{1}\z@{.7\linespacing\@plus\linespacing}{.5\linespacing}{\normalfont\scshape\centering\S}}
\makeatother

\usepackage{amsmath,amssymb,amsthm}
\usepackage{mathrsfs}
\usepackage{dsfont}
\usepackage{mathabx}
\usepackage{xparse}
\usepackage{microtype}

\usepackage{xcolor}
\usepackage[backref=page]{hyperref}
\hypersetup{%
    colorlinks,
    linkcolor={red!60!black},
    citecolor={green!60!black},
    urlcolor={blue!60!black}
}

\usepackage{varioref}

\usepackage[english,algoruled,vlined,resetcount,linesnumbered]{
algorithm2e}

\usepackage{bookmark}

\usepackage[abbrev,msc-links,backrefs]{amsrefs}
\usepackage{doi}

\renewcommand{\PrintDOI}[1]{\doi{#1}}

\usepackage[T1]{fontenc}
\usepackage{lmodern}

\usepackage[english]{babel}

\usepackage{tikz}

\usepackage{geometry}
\usepackage{enumitem}

\let\polishlcross=\l
\def\l{\ifmmode\ell\else\polishlcross\fi}

\renewcommand{\backslash}{\smallsetminus}

\makeatletter
\def\moverlay{\mathpalette\mov@rlay}
\def\mov@rlay#1#2{\leavevmode\vtop{%
   \baselineskip\z@skip \lineskiplimit-\maxdimen
   \ialign{\hfil$\m@th#1##$\hfil\cr#2\crcr}}}
\newcommand{\charfusion}[3][\mathord]{
    #1{\ifx#1\mathop\vphantom{#2}\fi
        \mathpalette\mov@rlay{#2\cr#3}
      }
    \ifx#1\mathop\expandafter\displaylimits\fi}
\makeatother

\newtheoremstyle{case}{}{}{\normalfont}{}{\itshape}{:}{ }{}

\newtheorem{thm}[section]{Theorem}

\newtheorem{ques}[equation]{Question}

\theoremstyle{definition}

\newtheoremstyle{case}{}{}{\normalfont}{}{\itshape}{\normalfont:}{ }{}

\theoremstyle{case}

\numberwithin{equation}{section}

\let\epsilon\varepsilon
\let\:\colon
\let\subset\subseteq

\def\({\left(}
\def\){\right)}
\def\[{\left[}
\def\]{\right]}
\let\<\langle
\let\>\rangle

\usepackage{datetime}
\usepackage{lineno}
\newcommand*\patchAmsMathEnvironmentForLineno[1]{%
\expandafter\let\csname old#1\expandafter\endcsname\csname #1\endcsname
\expandafter\let\csname oldend#1\expandafter\endcsname\csname end#1\endcsname
\renewenvironment{#1}%
{\linenomath\csname old#1\endcsname}%
{\csname oldend#1\endcsname\endlinenomath}}%
\newcommand*\patchBothAmsMathEnvironmentsForLineno[1]{%
\patchAmsMathEnvironmentForLineno{#1}%
\patchAmsMathEnvironmentForLineno{#1*}}%
\AtBeginDocument{%
\patchBothAmsMathEnvironmentsForLineno{equation}%
\patchBothAmsMathEnvironmentsForLineno{align}%
\patchBothAmsMathEnvironmentsForLineno{flalign}%
\patchBothAmsMathEnvironmentsForLineno{alignat}%
\patchBothAmsMathEnvironmentsForLineno{gather}%
\patchBothAmsMathEnvironmentsForLineno{multline}%
}

\begin{document}

\title{Ramsey equivalence of $K_n$ and $K_n+K_{n-1}$\\ for multiple colours}
\author{
Damian Reding
}

\shortdate
\yyyymmdddate
\settimeformat{ampmtime}
\date{\today, \currenttime}

\address{Technische Universit\"at Hamburg, Institut f\"ur Mathematik, Hamburg, Germany}
\email{damian.reding@tuhh.de}

\maketitle

\begin{abstract}
In 2015 Bloom and Liebenau proved that $K_n$ and $K_n+K_{n-1}$ possess the same $2$-Ramsey graphs for all $n\geq 3$ (with a single exception for $n=3$). In the following we give a simple proof that $K_n$ and $K_n+K_{n-1}$ possess the same $r$-Ramsey graphs for all $n, r\geq 3$.

\end{abstract}

\vspace{0.5cm}

Given an integer $r\geq 2$, an \emph{$r$-Ramsey graph $G$ for $H$} is such that every edge-colouring of $G$ with $r$ colours admits a monochromatic copy of $H$. Graphs $G_1$ and $G_2$ are \emph{$r$-Ramsey equivalent} if every $r$-Ramsey graph for $G_1$ is an $r$-Ramsey graph for $G_2$, and vice versa. For $r=2$ the concept was introduced in~\cite{SZZ}, where it was proved that $K_n$ and $K_n+K_{n-2}$ are $2$-Ramsey equivalent. In~\cite{BL} this was later improved to $K_n$ and $K_n+K_{n-1}$ for $n\geq 4$ and also (re)proved that $K_6$ is the only obstacle in the case $n=3$. The following extends these results to any number $r\geq 3$ of colours, for which it guarantees the non-existence of any obstacles.

\begin{thm}
$K_n$ and $K_n+K_{n-1}$ are $r$-Ramsey equivalent for all $n, r\geq 3$.
\end{thm}

\begin{proof}
To begin with, observe that for $n, r\geq 3$ the $(r-1)$-colour Ramsey numbers satisfy
$$R_{r-1}(n, \ldots, n)>rn\indent\text{whenever}\; (n, r)\neq (3, 3);$$
for fixed $n\geq 3$ this follows from $R_r(n)\geq R_{r-1}(n)+n$  by induction on $r$ (this simple inequality is a special case of one observed in~\cite{R}); note that for $n=3$ the induction beginning holds at $r=4$: $R_3(3)=17>4\cdot 3$, while for $n\geq 4$ it holds at $r=3$: $R_2(n)>3n$ (which in turn follows inductively from the relation $R_2(n+1)\geq R_2(n)+n$ and $R_2(4)=18>3\cdot 4$).

Now let $G$ be an $r$-Ramsey graph for $K_n$, where $(n, r)\neq (3, 3)$ and suppose that there is an $r$-edge-colouring of $G$ without a monochromatic $K_n+K_{n-1}$. We fix such a colouring and, adapting a method of~\cite{SZZ}, recolour a part of $G$ so as to obtain a contradiction to $r$-Ramseyness. W.l.o.g. there is a copy of $K_n$ in colour $r$; let $V_r$ denote its vertex set. For every $i\in [r-1]$ let $V_i$ denote the vertex set of a largest $i$-coloured clique in $V\backslash V_r$ of size at most $n$.

Consider the bipartition of $V(G)=A\cup B$, where $A:=V_1\cup\ldots V_{r-1}\cup V_r$ and $B:=V(G)\backslash A$. Note that $\left|A\right|\leq rn$, and also that for $n, r\geq 3$ we have $R_r(n, \ldots, n, 2)=R_{r-1}(n, \ldots, n)>rn$
for $(n, r)\neq (3, 3)$. Hence we can recolour the edges of $G[A]$ using colours $1,\ldots, r-1$ in such a way that there is no monochromatic $K_n$. We further recolour all the edges from $A$ to $B$ with colour $r$. We claim that the resulting $r$-edge-colouring contains no monochromatic $K_n$.

Indeed, if there were such copy in colour $i\in [r-1]$, say, it would need to lie in $G[B]$. But then $\left|G[V_i]\right|=n$, thus contradicting the non-existence of an $i$-coloured $K_n+K_{n-1}$ in the original colouring of $G$. Alternatively, if there were such in colour $r$, it would be using at most one vertex from $A$ (since after our recolouring there are no edges of colour $r$ in $G[A]$). Then, however, $G[B]$ would have to contain a $K_{n-1}$ in colour $r$, resulting in a similar contradiction.\\

It remains to prove the case $(n, r)=(3, 3)$. Towards this purpose we make three further preliminary observations about an arbitrary $3$-colour Ramsey graph $G$ for $K_3$:

1. Note that the chromatic number of $G$ satisfies $\chi(G)\geq R_3(3)=17$. This is dealt with by a simple argument due to Chvatal, which works equally well in the multicolour setting.

2. If a $3$-edge-colouring of $G$ contains a monochromatic triangle in every colour, then it contains a monochromatic copy of $K_3+K_2$: Let $V_0$ be the set of the vertices belonging to the three monochromatic triangles. Note that $G[V(G)\backslash V_0]$ contains an edge (otherwise $\chi(G)\leq\chi (G[V_0])+1\leq 10$, thus contradicting observation $1$). This edge must then form a monochromatic $K_3+K_2$ along with one of the three monochromatic triangles.

3. If $G$ (uncoloured) contains a copy of $K_6$, say $K$, then $G$ contains a further vertex-disjoint copy of $K_3$: Otherwise $G$ would vertex-decompose into $K$ and a triangle-free subgraph $F$ and we show that this is in fact impossible by giving $G$ a $3$-edge-colouring without a monochromatic triangle: let $v\in V(K)$ and edge-colour $K-v$ red-blue without monochromatic triangles (i.e. into a red and a blue $C_5$). Colour the remaining star of $K$ yellow and also colour $F$ blue. Finally, colour all edges between $K-v$ and $F$ yellow and all those between $v$ and $F$ red.

We are now able to provide the proof for $(n, r)=(3, 3)$. Fix a red-blue-yellow edge-colouring of a $3$-Ramsey graph $G$ for $K_3$ and let $R, B, Y$ denote the colour classes on $V(G)$, regarded as uncoloured subgraphs. We aim to find a monochromatic $K_3+K_2$ in $G$.

Suppose that \emph{none} of the subgraphs of $G$ formed by the union of any two of $R, B, Y$ is a $2$-Ramsey graph for $K_3$. Then the subgraph $R\cup B$ can be recoloured red-blue without monochromatic $K_3$'s. Hence there must exist a yellow $K_3$ in $Y\subset G$. Similarly we show that there is also both a blue and a red copy of $K_3$ in $G$. We are then done by observation 2.

Suppose that wlog. $R\cup B$ is a $2$-Ramsey graph for $K_3$. If $R\cup B$ does not contain a $K_6$, then it is a $2$-Ramsey graph for $K_3+K_2$ (see e.g.~\cite{BL}) and hence admits a monochromatic $K_3+K_2$. Otherwise, if $R\cup B$ does contain a copy of $K_6$, then, unless it contains a monochromatic $K_3+K_2$, that copy $K$ of $K_6$ contains both a red and a blue $K_3$. By observation $3$, $G$ then also contains a copy of $K_3$ that is vertex-disjoint from $K$. Then, if one of the edges of that $K_3$ is red or blue, a monochromatic $K_3+K_2$ is found; if not, then triangles of all three colours have been found, so we are done as before. This completes the proof.
\end{proof}

\emph{Concluding Remarks.} As the example of $K_3$ and $K_3+K_2$ shows, $3$-equivalence does not imply $2$-equivalence. This observation suggests the following two questions.

\begin{ques}
Are any two connected $3$-equivalent graphs necessarily $2$-equivalent?
\end{ques}

\begin{ques}
Are any two $2$-equivalent graphs necessarily $3$-equivalent?
\end{ques}

Note that an affirmative answer to the second question would guarantee that $2$-equivalence automatically forces $r$-equivalence for all $r\geq 3$; this follows inductively as in any colouring of an $r$-Ramsey graph with $r\geq 4$ either the graph formed by colour classes 1, 2 is $2$-Ramsey or the graph formed by colour classes $3,\ldots, r$ is $(r-2)$-Ramsey. The variant to look for multicoloured (instead of monochromatic) subgraphs in coloured graphs was considered in~\cite{HR}.

\bibliographystyle{abbrv}
\bibliography{requiv}

\end{document}